\documentclass[11pt]{amsart}
\usepackage{enumerate}
\usepackage{pb-diagram}
\usepackage{microtype}
\usepackage{amsmath}
\usepackage{amssymb,pb-diagram}
\usepackage[all]{xy}
\usepackage{graphicx} 
\usepackage{mathabx}
\usepackage{chapterbib}
\usepackage{epsfig}
\usepackage{amsfonts}
\usepackage{amsthm}
\usepackage{color}
\usepackage{latexsym}
\usepackage[pagebackref,colorlinks,citecolor=blue,linkcolor=blue,urlcolor=blue,filecolor=blue]{hyperref}
\usepackage[capitalise]{cleveref}

\input xy
\xyoption{all}
\usepackage{tikz}
\usetikzlibrary{arrows,shapes}
\newtheorem{theorem}{Theorem}[section]
\newtheorem{citing}[theorem]{Citation}
\newtheorem{lemma}[theorem]{Lemma}
\newtheorem{proposition}[theorem]{Proposition}
\newtheorem{corollary}[theorem]{Corollary}

\theoremstyle{definition}
\newtheorem{definition}[theorem]{Definition}
\newtheorem{example}[theorem]{Example}
\newtheorem{remark}[theorem]{Remark}

    \DeclareMathOperator {\Hom}{Hom}

   \newcommand{\Z}{\mathbb{Z}} 
    \newcommand{\Q}{\mathbb{Q}} 
\newcommand{\N}{\mathbb{N}} 
\newcommand{\R}{\mathbb{R}}

    \DeclareMathOperator {\im}{im}

    \DeclareMathOperator {\lcm}{lcm}

\newcommand{\quer}[1]{\overline{#1}}
\newcommand{\til}[1]{\widetilde{#1}}

\newcommand{\leer}{\varnothing} 

\newcommand{\mbf}{\mathbf}

\newcommand{\into}{\hookrightarrow}
\newcommand{\onto}{\twoheadrightarrow}
\newcommand{\nsgp}{\trianglelefteq} 

\newcommand{\veps}{\varepsilon} 
\newcommand{\tFm}[1]{\mathtt{F}_{#1}} 
\newcommand{\tFPm}[1]{\mathtt{FP}_{#1}}
\newcommand{\tFPn}{\mathtt{FP}_{n}}
\newcommand{\tFn}{\mathtt{F}_{n}}

\numberwithin{equation}{section}

\author[L.~Molyneux]{Lewis Molyneux}

\author[B.~Nucinkis]{Brita Nucinkis}
\address{
Royal Holloway, University of London\\ 
Department of Mathematics, McCrea Building, TW20~0EX Egham, UK}
\email{\href{mailto:lewis.molyneux@rhul.ac.uk}{lewis.molyneux@rhul.ac.uk}}
\email{\href{mailto:brita.nucinkis@rhul.ac.uk}{brita.nucinkis@rhul.ac.uk}}

\author[Y.~Santos]{Yuri Santos Rego}
\thanks{\noindent While working on this project at the OvGU Magdeburg, YSR was partially supported by the \emph{Deutsche Forschungsgemeinschaft} (DFG, German Research Foundation), 314838170, GRK 2297 \emph{MathCoRe}.}
\address{University of Lincoln\\
Charlotte Scott Research Centre for Algebra, School of Mathematics and Physics\\
Brayford Pool, LN6 7TS Lincoln, UK}
\email{\href{mailto:ysantosrego@lincoln.ac.uk}{ysantosrego@lincoln.ac.uk}}

\begin{document}

\title[Sigma invariants for $F_\tau$]{The Sigma invariants for the Golden Mean Thompson group}

\begin{abstract}
We use a method of Bieri, Geoghegan and Kochloukova to calculate the BNSR-invariants for the irrational slope Thompson's group $F_{\tau}$. To do so we establish conditions under which the Sigma invariants coincide with those of a subgroup of finite index, addressing a problem posed by Strebel.
\end{abstract}

\maketitle

\section{Introduction}

The study of what is now known as the first Sigma invariant or Bieri--Neumann--Strebel invariant $\Sigma^1(G)$ for a finitely generated group $G$ goes back to \cite{BieriStrebelMetabelianSigma,BNS} and was later extended by Bieri and Renz to a sequence of homotopical invariants 
\[ \cdots \subseteq \Sigma^n(G) \subseteq \Sigma^{n-1}(G) \subseteq \cdots \subseteq \Sigma^1(G) \subseteq S(G) \]
and homological invariants
\[ \cdots \subseteq \Sigma^n(G, R) \subseteq \Sigma^{n-1}(G,R) \subseteq \cdots \subseteq \Sigma^1(G,R) \subseteq S(G),\]
where $R$ is a commutative ring; cf. \cite{RenzThesis, Homology}.

In this note we compute the Sigma invariants for the Golden Mean Thompson group $F_\tau$ defined by Cleary in \cite{Cleary}, see also \cite{Ftau}. We prove:

\begin{theorem}\label{main-thm} Let $\lambda, \rho: F_\tau \to \R$ be the characters given by:
\[\lambda(f) =\log_{\tau}(f'(0)) \quad \text{ and } \quad \rho(f) = \log_{\tau}(f'(1)).\]
Then the Sigma invariants of $F_{\tau}$ are as follows:
\begin{enumerate}
\item $\Sigma^1(F_{\tau}) =  \Sigma^1(F_{\tau}, \mathbb{Z}) = S(F_{\tau}) \, \setminus \, \{[-\lambda], [-\rho]\}$, and
\item $\Sigma^{\infty}(F_{\tau}) =\Sigma^{\infty}(F_{\tau}, \mathbb{Z}) = \Sigma^2(F_{\tau}) = \Sigma^1(F_{\tau}) \, \setminus \, \{-[a\lambda +b\rho] \mid  a,b > 0\}.$
\end{enumerate}
\end{theorem}

Note that $\Sigma^1(F_\tau)$ was already known, see \cref{Sigma1-thm} below. {The computation of $\Sigma^1$ and higher Sigma invariants is of interest for various topological reasons; see, for instance, \cite{BNS, Homology, MMvW, SigmaF, StrebelNotes, KobanWong, DawidPolytopes, YashMattBNSR}. Particularly, we obtain the following information about coabelian subgroups of $F_\tau$.}

\begin{corollary} \label{cor:normal}
Let $N \nsgp F_\tau$ be a normal subgroup of homological type $\tFPm{2}$ for which the quotient $F_\tau / N$ is abelian. Then $N$ is of homotopical type $\tFm{\infty}$.
\end{corollary}

\begin{proof}
Immediate from \cref{main-thm}, \cite[Satz~C]{RenzThesis} and \cite[Theorem~B]{Homology}.
\end{proof}

\cref{main-thm} confirms that, similarly to the case of R. Thompson's original group $F$~\cite{SigmaF}, the Sigma invariants of $F_\tau$ are determined by an integral polytope (in the sense of \cite{DawidPolytopes}). The same behaviour is seen in other Thompson groups that `resemble' $F$ (e.g., \cite{SigmaF,ZaremskySigmaThompsonHoughton,braidedVZaremskyGeometry,YashMattBNSR}), though not all of them; see \cite{F23}. 

{While no unexpected phenomenon for the Sigma invariants of $F_\tau$ is observed, their computation slightly diverges from those in the above mentioned works. More precisely,} as a first step we consider the behaviour of the Sigma invariants $\Sigma^n(G)$ under passage to subgroups of finite index {--- which, to our knowledge, was not needed so far for other Thompson groups}. Using this, the computations for the Sigma invariants for $F_\tau$ then follow from methods similar to those of Bieri--Geoghegan--Kochloukova in \cite{SigmaF}.

Throughout denote by $G$ a finitely generated group, and by $G_{ab} \cong H_1(G;\Z)$ its abelianisation. We consider nontrivial characters  $\chi \in \Hom(G, \R) \cong H^1(G;\R)$. Define an equivalence relation by $\chi \sim \chi'$ if and only if there exists an $a \in \R_{>0}$ such that $\chi = a\chi'.$ The set of equivalence classes is a sphere in $\R^n$, called the {\emph character sphere} $S(G)$. Its dimension is determined by the torsion-free rank $r_0(G_{ab})$ of $G_{ab}$ {(equivalently, the first Betti number $b_1(G)$ of the group $G$)} and given by $r_0(G_{ab})-1$; see \cite[Lemma 1.1]{Bieri}. 
Now consider the following subset of the Cayley graph $\Gamma(G)$ with respect to some finite generating set: 
 $\Gamma_\chi(G)$ is the subgraph of $\Gamma(G)$ consisting of those vertices with $\chi(g) >0,$ and edges that have both initial and terminal vertices in $\Gamma_\chi(G).$ The first homotopical Sigma invariant is now defined as
\[\Sigma^1(G) = \{ [\chi] \in S(G) \mid \Gamma_\chi(G) \text{ is connected}\}.\]

Note that this is independent of the choice of generating set for the Cayley graph \cite{BNS}. For certain groups of homeomorphisms of the real line, including Thompson's group $F$ and the Golden Mean Thompson's group $F_\tau$ we have a complete description of $\Sigma^1(G)$:

\begin{citing}[{\cite[Chapter IV, Corollary 3.4]{Bieri}}] \label{Sigma1-thm}
Let $G$ be an irreducible subgroup of the group of piecewise linear homeomorphisms of the interval $[0,1]$. Take the characters $\chi_1(g)=ln(g'(0))$ as the natural log of the right derivative of an element $g \in G$ at $0$ and $\chi_2(g) = ln(g'(1))$ as the natural log of the left derivative of that element at $1$. If $G = \ker \chi_1 \cdot \ker \chi_2$, then $ \Sigma^1(G)^c = \{[\chi_1],[\chi_2]\}$.
\end{citing}

In the 1990s, Bieri and Strebel gave a formula to compute the complement $\Sigma^1(G)^c$ using $\Sigma^1(H)^c$ and a subsphere of $S(H)$ in case $H$ is a subgroup of finite index in $G$; see \cite[Chapter~III, Proposition~2.9]{Bieri} and \cite[Proposition~B1.11]{StrebelNotes}. In higher dimensions, a related formula was recently considered by Koban--Wong in \cite{KobanWong}. In his notes \cite[Section~B1.2c]{StrebelNotes}, Strebel goes on to wonder about the applicability of this formula, and poses the following.

\begin{citing}[{\cite[Problem~B1.13]{StrebelNotes}}]
Find situations where one is interested in $\Sigma^1(G)$ with $G$ admitting a subgroup of finite index which is easier to deal with and for which $\Sigma^1$ can be computed.
\end{citing}

We give a positive contribution towards Strebel's problem and find a sufficient condition for `equality' of Sigma invariants with those of subgroups of finite index. 

\begin{theorem}\label{htpy-fi}
Let $G$ be a group of type $\tFn$ with $H \leq G$ a subgroup of finite index and write $\iota : H \into G$ for the inclusion. If $r_0(G_{ab}) = r_0(H_{ab})$, then $\iota^\ast : S(G) \to S(H)$ is a well-defined homeomorphism and for all $n$ it holds
\[\iota^\ast(\Sigma^n(G)) = \Sigma^n(H).\]
\end{theorem}

\begin{theorem}\label{hom-fi}
Let $A$ be a $\Z G$-module of type $\tFPn$. Suppose $H \leq G$ is a subgroup of finite index and write $\iota : H \into G$ for the inclusion. If $r_0(G_{ab}) = r_0(H_{ab})$, then $\iota^\ast : S(G) \to S(H)$ is a well-defined homeomorphism and for all $n$ it holds
\[\iota^\ast(\Sigma^n(G,A)) = \Sigma^n(H,A).\]
\end{theorem}

Recalling the definition of $\iota^\ast$, any character $\chi \in \Hom(G,\R)$ can be restricted to a character of $H$, and we set $\iota^\ast(\chi) = \chi_{|H} \in \Hom(H,\R)$. In general, this map does not induce a function between character spheres. 
Thus the above statements also mean 
that the assignment $\iota^\ast([\chi]) = [\chi_{|H}]$ can be made on the level of character spheres, and we abuse notation also denoting this map by $\iota^\ast : S(G) \to S(H)$. 
We refer the reader to \cref{3} for the proofs of Theorems~\ref{htpy-fi} and~\ref{hom-fi} --- the main issue, as should be known to experts, is whether characters of the subgroup $H$ can be extended to characters of the whole group $G$. Examples~\ref{ex:1}, \ref{ex:2}, and~\ref{ex:nonextendedcharacter} illustrate how the equalities $\iota^\ast(\Sigma^n(G)) = \Sigma^n(H)$ and $\iota^\ast(\Sigma^n(G,A)) = \Sigma^n(H,A)$ can fail. 

We note that the problems of extending characters and of computing Sigma invariants from those of given subgroups appear in various guises in the literature; see, for example, \cite{Bieri,MMvW,KobanWong,DesiLuisBNSRBB,FriedlVidussi,DesiVidussi}. However, we were unable to find explicit references of statements along the lines of Theorems~\ref{htpy-fi} and~\ref{hom-fi}. 
We will make use of the homological result in \cite{MMvW} (included here as \cref{Prop9.3MMvW} in \cref{3}) during our proof of \cref{hom-fi}. 

Expanding on some related work, the authors in \cite{KobanWong} study Sigma invariants for finite-index normal subgroups $N \trianglelefteq G$, obtaining the image of $\Sigma^n(G)$ as an intersection of $\Sigma^n(N)$ with a certain subset of $\Hom(N,\R)$ invariant under a $G/N$-action. In \cite{DesiLuisBNSRBB}, extensions of characters from coabelian normal subgroups play a central role. More recently in \cite{FriedlVidussi,DesiVidussi}, the authors give conditions under which one can extend a character from certain normal subgroups of {infinite} index. 
Our formulation of Theorems~\ref{htpy-fi} and~\ref{hom-fi} gives a simple, easy-to-check condition on the Sigma invariants for (not necessarily normal) finite-index subgroups. We also remark that, over $\Z$ or a field, \cref{hom-fi} can be alternatively proved using techniques from Novikov homology and a recent generalisation of Sikorav's theorem due to Fisher \cite{SamFImproved}; see \cref{rmk:Dawid}. Our proof, in turn, uses only elementary methods. 

We stress that neither the equality $b_1(G) = r_0(G) = r_0(H) = b_1(H)$ nor finite index alone suffice as hypotheses, as the following examples show.

\begin{example} \label{ex:1}
Note that $r_0(G) = r_0(H)$ is insufficient to show an embedding of character spheres via $\iota^\ast$. As a counterexample, consider Thompson's original group $F = \langle x_0, x_1, \ldots \mid x_i^{-1} x_j x_i = x_{j+1} \text{ for } 0 \leq i < j \rangle$ and the subgroup $F[1] = \langle x_1, x_2, \ldots \mid x_i^{-1} x_j x_i = x_{j+1} \text{ for } 1 \leq i < j \rangle$. Clearly $F \cong F[1]$, and so $r_0(F) = r_0(F_1) = 2$. But any character $\chi \in \Hom(F,\mathbb{R})$ with $\chi(x_1) = 0$ restricts to the trivial character on $F[1]$, and all other character classes in $S(F)$ restrict to $[\pm \chi_1]$, where $\chi_1(x_0) = 0, \chi_1(x_1)=1$. Hence $\iota^\ast$ is only defined on a proper subset of $S(F)$, and the character classes in $\Sigma^n(F)$ on which $\iota^\ast$ is defined are mapped to a proper subset of $\Sigma^n(F[1])$.
\end{example}

\begin{example} \label{ex:2}
Similarly, $\vert G : H\vert < \infty$ alone does not guarantee the existence of a bijection between Sigma invariants of $G$ and $H$. For instance, the infinite dihedral group $D_\infty \cong \Z \rtimes C_2$ contains $\Z$ as a subgroup of index two. While $S(\Z) = \Sigma^1(\Z)$ is the $0$-sphere (and thus consists of two points), one has that $S(D_\infty)$ --- thus also $\Sigma^n(D_\infty)$ --- is empty as the abelianisation of $D_\infty$ is finite. 
\end{example}

Also note that the implications in Theorems~\ref{htpy-fi} and~\ref{hom-fi} cannot be reversed:

\begin{example} \label{ex:3}
Let $\mathbb{F}_n$ denote the free group on $n$ letters. It is known, see \cite[Proposition III.4.5]{Bieri}, that $\Sigma^1(\mathbb{F}_n) = \leer$ for all $n \geq 2$. Furthermore, $\mathbb{F}_n$ embeds with finite index in $\mathbb{F}_2$ \cite[Proposition I.3.9]{LyndonSchupp}. However, the torsion-free ranks of these groups are not equal as long as $n > 2$. 
\end{example}

We begin by establishing facts about both the Sigma invariants and $F_{\tau}.$ In \cref{3} we prove Theorems~\ref{htpy-fi} and~\ref{hom-fi}. And finally, in \cref{4} we compute the Sigma invariants for $F_\tau$.

\subsection*{Acknowledgments}

All three authors thank the Forschungsinstitut f\"ur Mathematik at the ETH Z\"urich for their hospitality during the Spring Semester 2023. The authors thank the anonymous referee for the very careful reading of an earlier version of the paper and the many helpful corrections and suggestions. We are also indebted to Conchita Mart\'inez-P\'erez for pointing us to the article by Kochloukova and Vidussi, and to Matthew Brin, Dawid Kielak, and Eduard Schesler for helpful comments.

\section{Background}

\subsection{Higher homotopical Sigma invariants}\label{higher}

We will begin with recalling some general definitions and facts that can be found, for example, in~\cite{Topics}. An Eilenberg--MacLane space, denoted $K(G,1)$ is an aspherical CW-complex $Y$ with $\pi_1(Y)=G.$ Its universal cover $X$ is contractible and has $G$ acting freely by deck transformations. Such a space is also called a model for $EG$ and is unique up to $G$-homotopy. A group $G$ is said to be of type $\tFn$ if there is a model for $EG$ with finite $n$-skeleton modulo the $G$-action, in which case we also say that this model has $G$-finite $n$-skeleton. Finally, $G$ is said to be of type $\tFm{\infty}$ if it is of type $\tFn$ for all $n \in \N$.

From now on let $G$ be of type $\tFn$ and let $X$ be a model for $EG$ with $G$-finite $n$-skeleton. 
The following construction is due to Renz \cite[Kapitel~II, Abschnitt~2]{RenzThesis}, see also \cite[Appendix~B, Section~B1.1]{Bieri}: 
 For a given character $\chi \in \Hom(G, \R)$ one defines an action of $G$ on $\R$ by $g\cdot r=r +\chi(g)$ for all $g \in G$ and $r \in\R$, which can be extended to a corresponding continuous $G$-equivariant map $h_\chi :  X \to \R$, also called a height function. Any such height function gives rise to an $\R$-filtration of $X$ given by the closed subspaces $h^{-1}_{\chi}([r, \infty))$. We shall consider ${X}_{h_\chi}^{[r,+\infty)}$, defined as the largest subcomplex of $X$ such that
\[x \in {X}_{h_\chi}^{[r,+\infty)} \implies h_\chi(x) \in [r,+\infty).\] 
When considering ${X}_{h_\chi}^{[0,+\infty)}$ we shall use the notation ${X}_{h_\chi}$.

\begin{definition}[{\cite[Kapitel~II, Definition~3.4]{RenzThesis} or \cite[Appendix~B, Definition in p.~194]{Bieri}}] \label{higher-top} 
Let $G$ be of type $\tFn$. Then the $n$-th Sigma invariant 
 $\Sigma^n(G) \subseteq S(G)$ is defined as follows:
 $[\chi] \in \Sigma^n(G)$ if there exists a model $X$ for $EG$ with $G$-finite $n$-skeleton and a corresponding height function $h_\chi$ on $X$ such that ${X}_{h_\chi}$ is $(n-1)$-connected.
\end{definition}

There are a priori different ways of extending the character $\chi$ to a $G$-equivariant height function $h_\chi$, though Renz shows that this distinction is immaterial and $\Sigma^n(G)$ is well-defined; cf. \cite[Kapitel~II, Bemerkungen~3.5]{RenzThesis}. This allows us to write $h$ instead of $h_\chi$ for an admissible height function extending a character $\chi$, if no confusion arises. 

While the connectivity condition in \cref{higher-top} might not hold for every model of $EG$ with $G$-finite $n$-skeleton, Renz~\cite{RenzThesis} also showed that the model may be arbitrary if one considers essential connectivity instead.

\begin{definition}[{\cite[Kapitel~II, Definition~3.6]{RenzThesis} or \cite[Appendix~B, Section~B1.2]{Bieri}}]
For $X^{[r,+\infty)}_h$ as defined above, we say that $X^{[r,+\infty)}_h$ is essentially $k$-connected for $k \in \Z_{\geq -1}$ if there is a real number $d \geq 0$ such that the map $\iota_j:\pi_j(X^{[r,+\infty)}_h)\to \pi_j(X^{[r-d,+\infty)}_h)$ induced by the inclusion $\iota : X^{[r,+\infty)}_h \into X^{[r-d,+\infty)}_h$ is the zero map for all $j \leq k$. 
\end{definition}

\begin{citing}[{\cite[Kapitel~IV, Satz~3.4]{RenzThesis} or \cite[Appendix~B, Theorem~B1.1]{Bieri}}] \label{essential} 
Let $G$ be a group of type $\tFn$ and let $X$ be an arbitrary model for $EG$ with $G$-finite $n$-skeleton. 
Let $\chi \colon G \to \R$ be a nontrivial character and $h \colon X \to \R$ a corresponding height function as above. Then 
\[ [\chi] \in \Sigma^n(G) \iff X_h \text{ is essentially } (n-1)\text{-connected}.\]
\end{citing}

\subsection{The homological invariant $\Sigma^n(G,A)$}\label{homSigma}

We will now give a brief overview of the definition and essential properties of the homological invariants $\Sigma^n(G,A)$, where $A$ is a $\Z G$-module, see \cite{Homology}. We follow the convention of Bieri, Renz, and Strebel \cite{BieriDimension,RenzThesis,Homology,Bieri} of working with left modules. In particular, a group or monoid acting on a module acts on the left.

\begin{definition}[{\cite[Chapter~VIII.4]{BrownKen}}]
Given a unital ring $R$, a (left) $R$-module $A$ is said to be of type $\tFPn$ over $R$ if it admits a resolution of the form 
\[\mbf{P} \quad \colon \quad \ldots \to P_i \to P_{i-1} \to \ldots \to P_1 \to P_0 \to A \to 0\] 
where the $P_i$ are free (left) $R$-modules which are finitely generated for $i \leq n$. In case $G$ is a group or monoid, we say that $G$ is of type $\tFPn$ if the trivial $\Z G$-module $A = \mathbb{Z}$ is {of} type $\tFPn$ over $R = \Z G$.
\end{definition}

One can analogously define `(right) type $\tFPn$', i.e., using right actions and right modules, and a group being of type $\tFPn$ does not depend on whether one works from the left or right; cf. \cite{BieriDimension}. However, the same is not true in the case of monoids (see, for instance, \cite{DesiConchitaSigmaBredon}), whence the importance of fixing a convention for the actions when working with both groups and monoids. 

\begin{definition}[{\cite[Section~1.3]{Homology}}]
Let $G$ be a group and $A$ a $\Z G$-module  of type $\tFPn$. The $n$-th homological Sigma invariant $\Sigma^n(G,A) \subseteq S(G)$ is defined as follows:
\[ [\chi] \in \Sigma^n(G,A) \iff A \text{ is of type } \tFPn \text{ over the subring } \mathbb{Z} G_{\chi} \subseteq \Z G,\] 
where $G_{\chi}$ is the submonoid $G_{\chi} = \{ g \in G \mid \chi(g) \geq 0\}$. 
\end{definition}

Now let $G$ be a group of type $\tFn$. Then (cf. \cite[Satz~B]{RenzThesis} and also \cite{Homology}) it holds:
\begin{equation} \label{homotopy}
\begin{aligned}
\Sigma^1(G) & = \Sigma^1(G; \Z) \\
\Sigma^n(G) & = \Sigma^2(G) \cap \Sigma^n(G; \Z ) \text{ for } n \geq 2
\end{aligned}
\end{equation}

Similarly to the homotopical case, it was shown that the definition of $\Sigma^n(G;A)$ does not depend on the partial finitely generated projective resolution of the $\Z G$-module $A$, see \cite[Theorem~3.2]{Homology}.

\subsection{Background on the Golden Mean Thompson group $F_{\tau}$} \label{sec:DefFtau}

Let  $\tau$ denote the small Golden Ratio, that is, the positive solution $\tau = \frac{\sqrt{5}-1}{2}$ to the equation $x^2+x=1$.

\begin{definition}[{\cite{Cleary}}]
The group $F_{\tau}$ is defined as the subgroup of piecewise linear, orientation-preserving homeomorphisms of the interval $[0,1]$ with slopes in the group $\langle \tau \rangle$ and breakpoints  in $\Z[\tau]$.
 \end{definition}

\begin{citing}[{\cite[Theorem~4.4]{Ftau}}]
$F_{\tau}$ has the {(infinite)} presentation
\begin{equation} \label{pres}
F_{\tau} \cong \langle x_i, y_i \mid a_j b_i = b_i a_{j+1}, y_i^2=x_i x_{i+1}; a,b \in \{x,y\}, 0 \leq i < j \rangle.
\end{equation}
\end{citing}

In the above, $i,j \in \Z_{\geq 0}$. We can write the generators of $F_\tau$ as functions on the interval $[0,1]$ in the following forms:

\begin{equation} \label{generators} 
\begin{aligned}
x_i(n) = & {\left\{
	\begin{array}{llll}
	n & \mbox{for } 0 \leq n \leq 1-\tau^i, \\
	\tau^{-2} n - \tau^{-1}(1-\tau^i)& \mbox{for } 1-\tau^i \leq n \leq 1-\tau^i+\tau^{i+4}, \\
	n + \tau^{i+3} & \mbox{for } 1-\tau^i+\tau^{i+4} \leq n \leq 1-\tau^{i+1}, \\
	\tau n + \tau^2  & \mbox{for } 1-\tau^{i+1} \leq n \leq 1,
	\end{array}
	\right.} \\
y_i(n) = & {\left\{
	\begin{array}{llll}
	n & \mbox{for } 0 \leq n \leq 1-\tau^i, \\
	\tau^{-1} n - \tau^{-1}(1-\tau^i)& \mbox{for } 1-\tau^i \leq n \leq 1-\tau^{i+1}, \\
	\tau n + \tau^2  & \mbox{for } 1-\tau^{i+1} \leq n \leq 1.
	\end{array}
	\right.}
\end{aligned}
\end{equation}

These elements can also be understood as equivalence classes of ordered tree-pairs, as described in \cite[Section~4]{Ftau}. 
As for the original Thompson group $F$, the elements of $F_\tau$ have a unique normal form \cite[Theorem~7.3]{Ftau}. We shall use the following normal form:
\begin{citing}[{\cite[Section~7]{Ftau}}] \label{seminormal}
Any element $f \in F_{\tau}$ can be uniquely expressed in the form 
\[ f= x_0^{i_0} y_0^{\epsilon_0} x_1^{i_1} y_1^{\epsilon_1} \cdots x_n^{i_n} y_n^{\epsilon_n} x_m^{-j_m} x_{m-1}^{-j_{m-1}} \cdots x_0^{-j_0}, \]
where $i_k, j_k \in \mathbb{Z}_{\geq 0}$, $\epsilon_k \in \{0,1\}$, $0 \leq k \leq n$, and moreover the following hold for all $k$:
\begin{enumerate}
\item If $i_k \neq 0 \neq j_k$, then at least one of $i_{k+1}$, $j_{k+1}$, $\epsilon_{k}$, $\epsilon_{k+1}$ is nonzero; 
\item  In case $f$ contains a subword of the form $x_k y_k x_{k+2} u x_{k+1}^{-1} x_k^{-1}$, then the middle subword $u$ contains a generator indexed either by $k+1$ or $k+2$.
\end{enumerate}
\end{citing}

Like $F$, the group $F_\tau$ also enjoys the strong homotopical and homological finiteness properties. 

\begin{citing}[{\cite{Cleary}}]\label{FtauFinfty}
The Golden Mean Thompson group $F_\tau$ is of type $\tFm{\infty}.$
\end{citing}

\section{Sigma invariants and finite index}\label{3}

In this section we prove Theorems~\ref{htpy-fi} and~\ref{hom-fi}. We begin by discussing, for $H \leq G$, maps between $H^1(H;\R) \cong \Hom(H,\R)$ and $H^1(G;\R) \cong \Hom(G,\R)$.

\begin{lemma} \label{lem:InjSurjHom}
Suppose $G$ is a finitely generated group, let $H \leq G$, and write $\pi : G \onto G_{\mathrm{ab}}$ for the canonical projection and $\iota : H \into G$ for the inclusion. Then the following hold.
\begin{enumerate}
\item \label{lem:InjSurjHom1} If $\vert G : H \vert < \infty$, then the map $\iota^\ast : \Hom(G,\R) \to \Hom(H,\R)$ induced by the inclusion is injective. 
\item \label{lem:InjSurjHom2} If the image $\pi(H)$ is infinite, then there exists a nontrivial morphism $e : \Hom(H,\R) \to \Hom(G,\R)$. That is, any character $\psi$ of $H \leq G$ gives rise to a character $e(\psi)$ of $G$ and the image $e(\Hom(H,\R)) \subseteq \Hom(G,\R)$ is a nonzero subspace.
\end{enumerate}
\end{lemma}

\cref{lem:InjSurjHom}\eqref{lem:InjSurjHom2} was observed by Kochloukova--Vidussi; cf. \cite[Proof of Theorem~1.1]{DesiVidussi}. Kochloukova and Vidussi work with characters in $G$ that are already assumed to be extensions of characters of a subgroup $H \leq G$. However, in the form we state \cref{lem:InjSurjHom}, the character $e(\psi) \in \Hom(G,\R)$ {need not} be a valid extension of the original character $\psi \in \Hom(H,\R)$. That is, it might be the case that $\iota^\ast \circ e (\psi) \neq \psi$; see \cref{ex:nonextendedcharacter} below.

From now on, when working in the abelianisation of a group, we will write the group operation additively.

\begin{proof}
Part~\eqref{lem:InjSurjHom1}: Take a nonzero character $\chi \in \mathrm{Hom}(G,\mathbb{R})$ and suppose  that $\iota^*(\chi)=\chi_{|H}=0$. As $\chi(G) \neq 0$, there exists $g \in G$ such that $\chi(g)\neq 0$, but as $\chi(H) = 0$ one has $g \not\in H$. Furthermore, we can say $g^n \not\in H$ for all $n \in \mathbb{N}$, as
\begin{equation*}
\begin{aligned}
g^n \in H \implies \chi(g^n) & = 0 \\
\iff n\chi(g) & = 0 \\
\iff \chi(g) & = 0,
\end{aligned}
\end{equation*}
contradicting $\chi(g) \neq 0$. Thus $g^n H$ are all distinct cosets of $H$, which means that $H$ is not finite index, contradicting our assumption. Hence, $\chi(H) \neq 0$.

Part~\eqref{lem:InjSurjHom2}: Consider the (finite dimensional) $\Q$-vector space $V = G_{\mathrm{ab}} \otimes_\Z \Q$. Since the image $\pi(H) \subseteq G_{\mathrm{ab}}$ is infinite, the set $\pi(H)$ contains some torsion-free elemenet and thus $\pi(H) \otimes_\Z \Q$ contains a partial basis for $V$, say $\mathcal{B}' = \{\overline{h}_1, \ldots, \overline{h}_m\}$, where each $\overline{h}_i$ is the image in $G_{\mathrm{ab}}$ of some $h_i \in H$. Extend this to a basis $\mathcal{B} = \{ \overline{h}_1, \ldots, \overline{h}_m, \overline{g}_{m+1}, \ldots, \overline{g}_{r} \}$ of $V$, again with $\overline{g}_j$ being the image of some $g_j \in G$. Since the image of characters of a group factors through their abelianisation, we may define 
\[ e(\psi)(g) := \sum_{i=1}^m a_i \psi(h_i), \]
where the $a_x$ with $x \in \mathcal{B}$ are the coordinates of the image of $g$ in $G_{\mathrm{ab}} \otimes_\Z \Q$ in the basis $\mathcal{B}$. It is straightforward to check that $e$ is a homomorphism from $\Hom(H,\R)$ to $\Hom(G,\R)$. Again because $\pi(H) \subseteq G_{\mathrm{ab}}$ is infinite and $G$ is finitely generated, the induced map $H \to \pi(H) \otimes_{\Z} \R \cong \R^m$ gives a nontrivial character, call it $\psi \in \Hom(H,\R)$, by projecting onto the line spanned by a nonzero vector of $\pi(H) \otimes_{\Z} \R$. By construction, the character $e(\psi) \in \Hom(G,\R)$ is also nontrivial.
\end{proof}

\begin{example} \label{ex:nonextendedcharacter}
As mentioned above, the proof of \cref{lem:InjSurjHom}\eqref{lem:InjSurjHom2} might yield an `extension' of a character $\psi$ of $H$ such that $\iota^\ast \circ e (\psi) \neq \psi.$ For example, let $H = \Z \times \Z \leq G = D_\infty \times \Z$ and take $\psi$ to be a character of $H$ which is nonzero on the first coordinate. 
\end{example}

However, provided $H$ is of finite index in $G$ and their first Betti numbers agree, one can always construct a lift from $\mathrm{Hom}(H,\mathbb{R})$ to $\mathrm{Hom}(G,\mathbb{R})$ that circumvents these problems. We summarise these properties in the following.

\begin{proposition} \label{Findex 2}
Let $G$ be a finitely generated group and  $H \leq G$ a subgroup of finite index. Then the following are equivalent: 
\begin{enumerate}
\item $r_0(G) = r_0(H)$.
\item $\iota^\ast : \Hom(G,\R) \to \Hom(H,\R)$ is an isomorphism of $\R$-vector spaces.
\item The assignment $\iota^\ast([\chi]) := [\chi_{|H}]$ is defined on all character classes $[\chi] \in S(G)$, and the corresponding map $\iota^\ast : S(G) \to S(H)$ is a homeomorphism.
\item Every character $\chi \in \Hom(H,\R)$ admits a lift $\chi' \in \Hom(G,\R)$ such that $\chi'|_H = \chi$ and $\chi \neq 0 \iff \chi' \neq 0$.
\end{enumerate}
\end{proposition}

\begin{proof}
The equivalences $(1) \iff (2) \iff (3)$ are immediate from \cref{lem:InjSurjHom}\eqref{lem:InjSurjHom1} as $\dim_\R(\Hom(\Gamma,\R)) = r_0(\Gamma)$ for any group $\Gamma$. Item~$(4)$ is equivalent to~$(2)$ as the function $e : \Hom(H,\R) \to \Hom(G,\R)$ given by $e(\chi) = \chi'$ is a right inverse to $\iota^\ast$.
\end{proof}

\begin{example} \label{ex:chareq}
It is not hard to explicitly construct the `extension map' $e : \Hom(H,\R)$ $\to \Hom(G,\R)$ of \cref{Findex 2}. Let $\{x_i,\ldots,x_n\}$ be a generating set for $G$ and write $r_0(G) = r_0(H) =k \leq n$. Without loss of generality one can assume that $\{\quer{x_1},\ldots,\quer{x_k}\}$ generates $(G_{ab})_0$. 
Since $|G:H| < \infty,$ for each $i=1,\ldots,n$, there exists an $\alpha_i \in \N$ such that $x_i^{\alpha_i} \in H$. Hence, using functoriality of  abelianisations,  and the fact that $\bar x_i$ has infinite order in $G_{ab}$, we have that $0 \neq\alpha_i \quer{x_i} \in  (H_{ab})_0$ for all $i =1,\ldots,k$.  Let $\alpha  = \lcm\{\alpha_1,\ldots,\alpha_k\}.$
Given a character $\chi: H \to \R$, we define its lift $e(\chi) = \chi': G \to \R$  by
\[\chi'(x_i) =\frac{1}{\alpha}\chi(x_i^\alpha), \text{ for all } i=1,\ldots,n.\]
\end{example}

To finish off \cref{hom-fi} we make use of the following:

\begin{citing}[{\cite[Proposition 9.3]{MMvW}}] \label{Prop9.3MMvW}
Suppose that $H \leq G$ is a subgroup of finite index and $A$ a $\Z G$-module of type $\tFPn,$ and suppose that $\chi: G \to \R$ restricts to a nonzero homomorphism of $H$. Then 
\[[\chi_{|H}] \in \Sigma^n(H, A) \iff [\chi] \in \Sigma^n(G,A).\]
\end{citing}

\begin{proof}[Proof of \cref{hom-fi}]
Immediate from \cref{Findex 2} and \cref{Prop9.3MMvW}.
\end{proof}

\begin{remark} \label{rmk:Dawid}
{In case $A = \Z$ or a field, \cref{hom-fi} can also be proved as follows: by change of rings \cite{BieriDimension} and noting that the Novikov ring $\widehat{A[G]}^{\chi}$ is isomorphic to the tensor product $\widehat{A[H]}^{\chi|_H} \otimes_{A[H]} A[G]$, an application of \cref{Findex 2} combined with the equivalence (1)~$\iff$~(5) from a result of Fisher~\cite[Theorem~5.3]{SamFImproved} proves the claim.}
\end{remark}

For completeness we now give an elementary proof of the homotopical part, which needs the following.

\begin{proposition} \label{sigmaeq} Let $G$ be a group of type $\tFn$ and $H$ a subgroup of finite index such that $r_0(G)=r_0(H).$ With the notation of \cref{Findex 2} we have
\[[\chi] \in \Sigma^n(G) \implies [\chi_{|H}] \in \Sigma^n(H),\]
and 
\[[\chi] \in \Sigma^n(H) \implies [\chi'] \in \Sigma^n(G).\]
\end{proposition}

\begin{proof} To prove the first claim, consider a model $X$ for $EG$ with $G$-finite $n$-skeleton. Now suppose $[\chi] \in \Sigma^n(G),$ hence $X^{[0,+\infty)}_{h_\chi}$ is $(n-1)$-connected for the height function $h_\chi$ corresponding to $\chi.$ Since $H$ is finite index in $G$, the space $X$ is also a model for $EH$ with $H$-finite $n$-skeleton, and $h_\chi =h_{\chi_{|H}}$. Hence $[\chi_{|H}] \in \Sigma^n(H)$.

\noindent Let us now assume  $[\chi] \in \Sigma^n(H)$. Again using $|G:H|<\infty$, choose a model for $X$ for $EH$ as above: $X$ is a simplicial complex with $G$-finite $n$-skeleton and one $G$-orbit of zero-cells labeled by $G$. 

We now fix a set $T=\{t_0,\ldots,t_{m-1}\}$ of coset representatives of $H$ in $G$, put $t_0 =e$, and construct an $H$-equivariant height function $h_\chi: X \to \R$ on the vertices of $X$ as follows: For $\gamma \in H$ we put $h_{\chi}(\gamma)=\chi(\gamma)$ and for $h_{\chi}(t_i)=0.$  Hence, since every $g \in G$ has a unique expression as $g =t_i \gamma$, we get
$$h_\chi(g) = h_{\chi}(t_i)+h_{\chi}(\gamma)=\chi(\gamma).$$
Finally, we extend this function piecewise linearly to the entire $n$-skeleton on $X.$
Hence $X^{[0,+\infty)}_{h_\chi}$ is essentially $(n-1)$-connected, see \cref{essential}.

It remains to show that this connectivity property remains true using a height function $h_{\chi'}$ corresponding to a lift $\chi'$ of $\chi$. 
Note that $\chi'(t_i)$ is not necessarily equal to $0$. Define $d=\textbf{min}\{\chi(t_i)\}$. 

We claim that, for very $g \in G$, $h_\chi(g) \geq 0$ if and only if $\chi'(g) \geq d$. To see this, write $g = t_i\gamma$ as above. Since  $h_\chi(g) = \chi(\gamma)$ and
$\chi'(g) = \chi'(t_i \gamma) = \chi'(t_i) +\chi(\gamma)$, we get 
\[h_\chi(g)  \geq 0 \iff \chi(\gamma) \geq 0 \iff \ \chi'(t_i) + \chi(\gamma) \geq d+0 \iff \chi'(g) \geq d,\] 
as required.

This now implies implies that the  $0$-skeleton of $X^{[0,+\infty)}_{h_\chi}$ is precisely the same as the $0$-skeleton of $X^{[d,+\infty)}_{h_{\chi'}}$. As the space $X^{[r,+\infty)}_{h_\chi}$ is defined as the maximal subcomplex of $X$ contained in $h_{\chi}^{-1}([r,+\infty))$, where an $m$-cell is included if all of its boundary cells are included \cite[Appendix~B, p.~194]{Bieri}, we have shown that $X^{[0,+\infty)}_{h_\chi} = X^{[d,+\infty)}_{h_{\chi'}}$. Hence $[\chi'] \in \Sigma^n(G)$ again by \cref{essential}.
\end{proof}

\begin{proof}[Proof of \cref{htpy-fi}]
This follows from  Propositions~\ref{Findex 2} and~\ref{sigmaeq}.
\end{proof}

\section{The Sigma invariants for $F_{\tau}$}\label{4}

We begin by collecting some properties of $F_\tau$ and its characters as well as exhibiting a finite index subgroup which satisfies the assumptions of Theorems~\ref{htpy-fi} and~\ref{hom-fi}.

It was shown in \cite[Chapter 5]{Ftau} that
$$(F_\tau)_{ab} \cong \Z^2 \oplus \Z/2\Z.$$
Hence 
$$S(F_\tau) =S^1.$$
Similarly to the original Thompson's group case, we have the two linearly independent characters $\lambda$ and $\rho$ given by some logarithm of the slopes at $0$ and $1$ respectively, such that
$[\lambda]$ and $[\rho]$ span $S(F_\tau)$. In particular, these, for every $f \in F_\tau$,  are given by
\[\lambda(f) =\log_{\tau}(f'(0)) \quad \text{ and } \quad \rho(f) = \log_{\tau}(f'(1)).\]

By taking appropriate subdivisions of $[0,1]$, one can construct elements $f \in F_\tau$ with support in $[0,b] \cap \Z[\tau]$ for some $b < 1$ and such that $f'(0)=\tau$. Analogously, one can find $g \in F_\tau$ with support in $[a,1]$ for some $a > 0$ and with $g'(1)=\tau$. Hence $\lambda(f) = 1 = \rho(g)$, $\lambda(g) = 0 = \rho(f)$ and thus $\lambda$ and $\rho$ are linearly independent.

As an example, we can use the following elements:

\begin{example}\label{li-example}
\[
f(x) = \left\{
	\begin{array}{lll}
	\tau x & \mbox{for } 0 \leq x \leq \tau^2 \\
	\tau^{-1}x - \tau^2 & \mbox{for } \tau^2 \leq x \leq \tau \\
	x  & \mbox{for } \tau \leq x \leq 1
	\end{array}
	\right.
\]
\[
g(x) = \left\{
	\begin{array}{lll}
	 x & \mbox{for } 0 \leq x \leq \tau^2 \\
	\tau^{-1}x - \tau^3 & \mbox{for } \tau^2 \leq x \leq \tau \\
	\tau x + \tau^2  & \mbox{for } \tau \leq x \leq 1.
	\end{array}
	\right.
\]
\end{example}

\begin{proposition}\label{K}
Let $K$ denote the subgroup of $F_\tau$ generated by \hfill\break $\{x_0,x_1,y_1,x_2,y_2,...\}.$ Then $|F_\tau: K| =2$ and 
$K_{ab} \cong \Z^2 \oplus \Z/2\Z.$
\end{proposition}

\begin{proof} We claim $F_\tau = K \sqcup y_0K.$ To do so, consider $g \in F_\tau$ in normal form, see \cref{seminormal}:

\[g=x_0^{i_0}y_0^{\epsilon_0}x_1^{i_1}y_1^{\epsilon_1}\cdots x_n^{i_n}y_n^{\epsilon_n}x_m^{-j_m}x_{m-1}^{j_{m-1}}\cdots x_0^{j_0}\]
where $i_0,\ldots,i_n, j_0,\ldots,j_m \in \Z_{\geq 0}$ and $\epsilon_0,\dots,\epsilon_m \in \{0,1\}$. Hence
\[gK = x_0^{i_0}y_0^{\epsilon_0}K.\] 
When $\epsilon_0 =0$ we have $g \in K,$ and when $\epsilon_0 =1$ a repeated application of the following computation gives $g \in y_0K$:
\begin{equation} \label{relator}
\begin{aligned}
x_0^{i_0}y_0 & = x_0^{i_0-1}x_0y_0 = x_0^{i_0-1}x_0x_1x_1^{-1}y_0 = x_0^{i_0-1}y_0^2x_1^{-1}y_0 = \\
& = x_0^{i_0-1}y_0^2y_0x_2^{-1} = x_0^{i_0-1}y_0y_0^2x_2^{-1} = x_0^{i_0-1}y_0x_0x_1x_2^{-1}. 
\end{aligned}
\end{equation}

Consider any word in the generators $x_i$ and $y_j$ $(i \geq 0, j \geq 1)$ in $F_\tau.$ The relations of $F_\tau$, see \cref{pres}, imply that in any other expression on this element, the occurrence of $y_0^k$ will have $k$ an even integer. Hence, in the normal form of \cref{seminormal} such an element will have no occurrence of $y_0$. This implies that $K$ is a proper subgroup of $F_\tau,$ and moreover $|F_\tau : K| =2$. 

To determine the abelianisation we do a similar calculation to that in \cite[Section~5]{Ftau}: We denote the images of an element $f \in F_\tau$ in the abelianisation by $\bar f$ and write the group operation additively. From the relations it follows immediately that $\bar{x}_i =\bar{x}_{i+1}$ and that $2\bar{y}_i = 2\bar{x_1}$ for all $i \geq 1.$ Substituting $\bar z = \bar{y}_1-\bar{x}_1,$ we have the two infinite order generators $\bar{x}_0$ and $\bar{x}_1$ as well an order $2$ generator $\bar z$ as required.
\end{proof}

\noindent
Let $H$ be a group and $\sigma : H \to H$ a monomorphism. An ascending HNN extension (with base $H$) is a group given by the presentation 
\[H\ast_{t,\sigma} = \langle H,t \mid tht^{-1}=\sigma(h); h \in H \rangle.\]

We now consider the subgroup $F_\tau[1] \leq F_\tau$ generated by $\{x_1,y_1,x_2,y_2,...\}$. In analogy to Thompson's $F$, there is a well-known monomorphism 
$\sigma: F_\tau \to F_\tau$ given by $\sigma(x_n)=x_{n+1}$ and $\sigma(y_n)=y_{n+1}$, whose image is clearly $F_\tau[1] \subsetneq F_\tau$. Restricting to $F_\tau[1]$ gives a monomorphism $\sigma : F_\tau[1] \to F_\tau[1]$ whose image is the proper subgroup $F_\tau[2] \subsetneq F_\tau[1]$ generated by $\{x_2,y_2,x_3,y_3,\ldots\}$, and so on. Hence any $F_\tau[m]$ is isomorphic to $F_\tau$ and thus of type $\tFm{\infty}$. 
Much like $F$ is an HNN extension over a copy of itself (cf. \cite[Proposition~1.7]{BrownGeoghegan}), the group $K$ --- which contains $F_\tau[1]$ by definition --- differs from its subgroup $F_\tau[1] \cong F_\tau$ by the stable letter $x_0$. 

\begin{lemma} \label{lem:HNN}
The subgroup $K \leq F_\tau$ is isomorphic to the HNN extension 
\[K \cong F_\tau[1]\ast_{t, \sigma} = \langle F_\tau[1], t \mid t g t^{-1} = \sigma(g); g \in F_\tau[1] \rangle\]
by mapping $t$ to $x_0^{-1}$ and $F_\tau[1]$ to itself.
\end{lemma}

\begin{proof}
For this proof we implicitly use standard facts about presentations and HNN extensions; cf. \cite[Chapter~IV, Section~2]{LyndonSchupp}. 

Let $\langle X \mid R \rangle$ denote the obvious presentation of $F_\tau[1]$, that is, the same as that of $F_\tau$ from \cref{pres} but with decorated generating set $X = \{\til{x_i}, \til{y_i} \mid i \geq 1\}$ and indices starting from $1$. The HNN extension $F_\tau[1]\ast_{t, \sigma}$ is thus given by the (abstract) group presentation 
\[ F_\tau[1]\ast_{t, \sigma} \cong L := \langle X, t \mid R, \, t \til{x_i} t^{-1} = \til{x}_{i+1}, \, t \til{y_i} t^{-1} = \til{y}_{i+1} \text{ for all } i \geq 1 \rangle.\]

The obvious map 
\[ \phi : L \to K \, \text{ induced by } \, t \mapsto x_0^{-1}, \, \til{x_i} \mapsto x_i, \, \til{y_i} \mapsto y_i \]
is a well-defined group homomorphism since all defining relations in $L$ hold in $K$. It is surjective by construction, and we want to check that it is also injective. Note that, since $L$ is an HNN extension, the group $F_\tau[1]$ effectively embeds in $L$ as its obvious subgroup $\langle X \rangle$. The restriction of $\phi$ to $\langle X \rangle$ is thus an isomorphism onto its image $F_\tau[1] \leq K$. In particular, if $g \in \langle X \rangle$, the isomorphisms $\langle X \rangle \cong F_\tau[1] \cong F_\tau$ and \cref{seminormal} yield a (unique) normal form for $g$ matching the (unique) normal form of $\phi(g) \in K \subseteq F_\tau$ (by dropping the tildes), and such a normal form of $\phi(g)$ in $K$ does not involve the generator $x_0$. 

Now let $w \in \ker(\phi) \trianglelefteq L$. As $L$ is an HNN extension, we may write $w$ in normal form 
\[w = g_0 t^{\veps_1} g_1 t^{\veps_2} \cdots g_{n-1} t^{\veps_{n-1}} g_n \]
with each $\veps_i \in \{\pm 1\}$ and $g_i \in \langle X \rangle$. If $\veps_i = -1$, repeated applications of the defining relations in $L$ yield 
\[g_i t^{-1} = t^{-1} g_i' \, \text{ for some } \, g_i' \in \langle \{ \til{x_j}, \til{y_j} \mid j \geq 2 \} \rangle \leq \langle X \rangle.\]
Similarly, if $\veps_i = 1$, then 
\[t g_i =  g_i' t \, \text{ for some } \, g_i' \in \langle \{ \til{x_j}, \til{y_j} \mid j \geq 2 \} \rangle \leq \langle X \rangle.\]
Thus, writing $a = \#\{ i \mid \veps_i < 0\} \geq 0$ and $b = \#\{ i \mid \veps_i > 0\} \geq 0$, the word $w$ can be rewritten as 
\[ w = t^{-a} g' t^b \, \text{ where } \, g' \in \langle X \rangle.\] 
As $g' \in \langle X \rangle$, we may replace it by its (unique) normal form in $\langle X \rangle \cong F_\tau[1]$, if necessary. Mapping over to $K$, we obtain 
\[\phi(w) = \phi(t)^{-a} \phi(g') \phi(t)^b = x_0^a \phi(g') x_0^{-b},\]
where the subword $\phi(g')$ lies in $F_\tau[1]$ and is written in its (unique) normal form, not involving the letter $x_0$. In particular, the word $x_0^a \phi(g') x_0^{-b} \in K \subseteq F_\tau$ is given in a normal form as in \cref{seminormal}. 

Suppose first that $x_0^a \phi(g') x_0^{-b}$ is already in normal form, see \cref{seminormal}. Since $1 = \phi(w) = x_0^a \phi(g') x_0^{-b}$ by assumption, the above considerations imply that $g' = 1$ and $a = b$, whence $w$ is trivial in $L$. 

If $x_0^a \phi(g') x_0^{-b}$ is not in normal form, then $\phi(g')$ has no occurrences of the letters $x_1$ or $y_1$. We can assume that $a \geq b$. Hence $x_0^a \phi(g') x_0^{-b}=x_0^{a-b}\phi(g'[b]),$ where $g'[b]$ denotes the word $g'$ with the indices of the $x_i$ and $y_i$ increased by $b.$ This is now in normal form as in \cref{seminormal}, and as above it means that $g'[b]=1,$ hence $g'=1,$ and that $a-b=0.$ Again, $w$ is trivial in $L.$
This finishes the proof.
\end{proof}

We can finally adapt the calculations for Thompson's group $F$ as in \cite{SigmaF} to compute the Sigma invariants for $F_\tau$.

\begin{citing}[{\cite[Theorem 2.1]{SigmaF}}] \label{points}
Let $G$ decompose as an ascending HNN extension $H\ast_{t,\sigma}$. Let $\chi$ be a character such that $\chi(H)=0$, $\chi(t)=1$.

\begin{itemize}

\item Suppose $H$ is of type $\tFn$, then $[\chi] \in \Sigma^n(G)$. 
\item Suppose $H$ is of type $\tFPn$, then $[\chi] \in \Sigma^n(G;\Z)$.
\item If $H$ is finitely generated and $\sigma$ is not surjective, then $[-\chi] \notin \Sigma^1(G)$.
\end{itemize}
\end{citing}

\begin{lemma} \label{lem:plusminuslambdarho} Let $\lambda$ and $\rho$ be the characters defined at the beginning of this section. Then
\[ [\lambda], [\rho] \in \Sigma^{\infty}(K) \cap \Sigma^{\infty}(F_\tau) \quad \text{ and } \quad  [-\lambda], [-\rho] \notin \Sigma^1(K) \cup \Sigma^1(F_\tau),\]
\[ [\lambda], [\rho] \in \Sigma^{\infty}(K;\Z) \cap \Sigma^{\infty}(F_\tau;\Z) \text{ and } [-\lambda], [-\rho] \notin \Sigma^1(K;\Z) \cup \Sigma^1(F_\tau;\Z).\]
\end{lemma}

\begin{proof} We begin by determining the result for $[\lambda]$ and $[-\lambda]$. The support of $F_\tau[1]$ lies in $[\tau^2, 1]$ and hence $\lambda(F_\tau[1]) = 0.$ The slope of $x_0$ at $0$ is $\tau^{-2}$. Hence, taking the character $\chi:=\frac{1}{2}\lambda \in [\lambda]$, we obtain $\chi(t) = 1$. We can thus apply \cref{points} to conclude that $[\lambda] \in \Sigma^\infty(K)$ and $[-\lambda] \notin \Sigma^1(K)$. 
By \cref{htpy-fi}, it follows that $[\lambda] \in \Sigma^\infty(F_\tau)$ and $[-\lambda] \notin \Sigma^1(F_\tau)$.

 As in \cite[Section 1.4]{SigmaF} we now consider a specific automorphism $\nu$ of $F_\tau$ to clear the case of $\rho$. Viewing the group $F_\tau$ as a group of PL homeomorphisms of the unit interval, $\nu$ is given by conjugation by 
$t \mapsto 1-t.$  This induces a homeomorphism of the character sphere that in particular swaps $[\lambda]$ with $[\rho]$, and also $[-\lambda]$ with $[-\rho]$,  thus proving the lemma for $F_\tau$.  A further application of \cref{htpy-fi} also yields the result for $K$. 

The homological variant of the lemma follows similarly; see also \cref{homotopy}.
\end{proof}

We shall now consider the arcs between $[-\lambda]$ and $[-\rho]$ on the character sphere $S(F_\tau)$. Since $[-\lambda]$ and $[-\rho]$ are not antipodal points, there is a unique (closed) geodesic segment in $S(F_\tau)$ connecting them, which we denote by $\mathrm{conv}([-\lambda],[-\rho])$. In the other direction, there is a unique local geodesic from $[-\lambda]$ and $[-\rho]$, which we call the \emph{long arc}, whose union with $\mathrm{conv}([-\lambda],[-\rho])$ yields the great circle in $S(F_\tau)$ containing $[-\lambda]$ and $[-\rho]$, in particular in this one-dimensional case, this is just  $S(F_\tau)$ itself. We will need the following:

\begin{citing}[{\cite[Theorem 2.3]{SigmaF}}] \label{long}
Let $G$ decompose as an ascending HNN extension $G=H*_{t,\sigma}$. Let $\chi$ be a character of $G$ such that $\chi_{|H} \neq 0$. If $H$ is of type $\mathtt{F}_{\infty}$ and $\chi|_H \in \Sigma^{\infty}(H)$, then $\chi \in \Sigma^{\infty}(G)$.
\end{citing}

\begin{proposition} \label{prop:almosteverything} All of $S(F_\tau)$, except possibly the closed geodesic 
\[\mathrm{conv}([-\lambda],[-\rho]),\] 
lies in $\Sigma^\infty(F_\tau)$ and in $\Sigma^\infty(F_\tau,\Z)$.
\end{proposition}

\begin{proof}
Again we use our previous expression of the subgroup $K$ as an HNN extension of $H=F_{\tau}[1]$. By \cref{lem:plusminuslambdarho}, we know that $[\rho] \in \Sigma^{\infty}(K) \cap \Sigma^{\infty}(F_\tau)$. Now let $\chi \in \Hom(F_\tau, \R)$ be arbitrary. We claim that 
\begin{equation} \label{claim:x1positive}
\chi(x_1) > 0 \iff \chi_{|H} \in [\rho|_H].
\end{equation} 
In effect, $\chi = r\lambda + s\rho$ for some (unique) $r,s \in \R$ as $\lambda$ and $\rho$ are linearly independent and $\dim_\R(\Hom(F_\tau,\R))=2$. Since $\lambda(x_1) = \lambda(y_1) = 0$ and $a_j = a_0 a_{j+1} a_0^{-1}$ for any $j \geq 1$ and $a \in \{x,y\}$, it follows that $\chi(w) = s \rho(w)$ for any $w \in H = F_\tau[1]$. This means that $\chi|_H \in \{[\rho|_H],[-\rho|_H]\}$. Finally, $\rho(x_1)=1$ implies that $\chi(x_1) = s$, whence $\chi(x_1) > 0$ if and only if $\chi_{|H} \in [\rho|_H]$. 

From here, we highlight that $H=F_{\tau}[1]$ is isomorphic to $F_{\tau}$, via the isomorphism $\gamma$ such that $\gamma(x_i)=x_{i-1}$ and $\gamma(y_i)=y_{i-1}$ for $i \geq 1$. The homeomorphism $S(F_\tau[1]) \cong S(F_\tau)$ induced by $\gamma$ maps $[\rho|_H]$ to $[\rho]$. As $[\rho] \in \Sigma^{\infty}(F_{\tau})$, this means $[\rho|_H] \in \Sigma^{\infty}(F_{\tau}[1])$. In particular, if $\chi \in \Hom(F_\tau,\R)$ is positive on $x_1$, Claim~\eqref{claim:x1positive} yields $[\chi_{|H}] = [\rho|_H] \in \Sigma^{\infty}(F_{\tau}[1])$. From here, we can apply \cref{long} to conclude that $[\chi_{|K}] \in \Sigma^{\infty}(K)$. Thus, $\chi(x_1)>0 \implies [\chi_{|K}] \in \Sigma^{\infty}(K)$, whence $[\chi] \in \Sigma^{\infty}(F_{\tau})$ by \cref{htpy-fi}.

A straightforward computation shows that any character $\chi$ on the straight line from $\lambda$ to $\rho$ in $\Hom(F_\tau,\R)$ satisfies $\chi(x_1) > 0$. The same holds for any character on the straight line from $\rho$ to $-\lambda$.
Hence, we have that the open arc in $S(F_{\tau})$ from $[\lambda]$ to $[-\lambda]$ that contains $[\rho]$ actually lies in $\Sigma^{\infty}(F_{\tau})$. Arguing again with the symmetry in $S(F_\tau)$ given by the automorphism $\nu$ induced by conjugation with $t \mapsto 1-t$, we conclude that the open arc from $[\rho]$ to $[-\rho]$ containing $[\lambda]$ is also in $\Sigma^{\infty}(F_\tau)$. Altogether, the long (open) arc from $[-\lambda]$ to $[-\rho]$ is in $\Sigma^{\infty}(F_{\tau})$, as claimed. The homological version follows directly from \cref{homotopy}. 
\end{proof}

\noindent It now remains to consider the remaining short arc $\mathrm{conv}([-\lambda],[-\rho])$. To do this we will follow the approach of \cite[Section 2.3]{SigmaF}. We need the following two results:

\begin{citing}[{\cite[Corollary 1.2]{SigmaF}}]\label{kernel}
The kernel of a nonzero discrete character $\chi$ has type $\tFPn$  over the ring $R$ if and only if both $[\chi]$ and $[-\chi]$ lie in $\Sigma^n(G,R)$. 
\end{citing}

\begin{citing}[{\cite[Theorem 2.7]{SigmaF}}] \label{FP2}
Assume $G$ contains no nonabelian free subgroups and is of type $\tFPm2$ over a ring $R$. Let $\widetilde{\chi}: G \to \R$ be a nonzero discrete character. Then $G$ decomposes as an ascending HNN extension $H*_{t, \sigma}$, where $H$ is a finitely generated subgroup of $\ker(\widetilde{\chi})$ and $\widetilde{\chi}(t)$ generates the image of $\widetilde{\chi}$.
\end{citing}

\begin{proposition} Let $R$ be a ring. Then 
\[\mathrm{conv}([-\lambda],[-\rho])\cap \Sigma^2(F_{\tau},R) = \leer.\]
\end{proposition}

\begin{proof}
It suffices to show that no discrete character $\chi \in \mathrm{conv}([-\lambda],[-\rho])$ lies in $\Sigma^2(F_{\tau},R)$ because such characters are dense in $\mathrm{conv}([-\lambda],[-\rho])$ and $\Sigma^2(F_{\tau},R)$ is open; see, e.g., \cite[Proposition~2.9]{SigmaF}. Observe further that $[-\lambda]$, $[-\rho] \notin \Sigma^2(F_{\tau},R)$ by \cref{lem:plusminuslambdarho}. 

So let $\chi$ be a discrete character of the form $\chi = a \lambda + b \rho$, with $a,b \in \mathbb{Q} \setminus\{0\}$. Using the elements $f,g \in F_\tau$ of \cref{li-example}, we can construct elements $t \in F_{\tau}$ with the following properties: 
\begin{equation} \label{eq:constructingt}
\lambda(t)=mb \quad \text{ and } \quad \rho(t)=-ma \text{ for some } m \in \mathbb{Q} \setminus \{0\}.
\end{equation}

In particular, $\chi(t) = 0$. Since $\lambda$ has discrete image in $\R$ and $a \neq 0$, there exists $t_0$ satisfying condition~\eqref{eq:constructingt} such that $|\lambda(t_0)|$ is minimal among all elements $t$ fulfilling the properties listed in~\eqref{eq:constructingt}. Moreover, $\lambda(t_0) \neq 0$ for otherwise $t_0$ would not fulfill~\eqref{eq:constructingt}.

Let $G = \ker(\chi)$. Then, since the abelianisation of $F_\tau$ is $\Z^2 \times \Z/2\Z,$ we have that $G = \langle \sqrt{F'_\tau}, t_0 \rangle = \sqrt{F'_\tau} \rtimes \langle t_0 \rangle$, where $\sqrt{F'_{\tau}} := \{f \in F_{\tau} \mid f^n \in F_{\tau}' \text{ for some } n\}$. Note that $\lambda|_G$ is a discrete nonzero character vanishing on the subgroup $\sqrt{F'_\tau} \leq G$ and such that $\im(\lambda|_G)$ is generated by $\lambda(t_0)$. 

Now suppose $G$ has type $\tFPm2$ over a ring $R$. By \cref{FP2},  we can decompose $G$ as the HNN extension $H*_{t, \sigma}$, where $H$ is a finitely generated subgroup of $\sqrt{F'_{\tau}}$. As $H$ is generated by a finite set of elements of $F_{\tau}$, and each generator has support away from $0$ and $1$, there exists a value $\varepsilon'' > 0$ such that all elements of $H$ are supported in the interval $[\varepsilon'', 1-\varepsilon'']$. Similarly, as $t_0$ has finitely many breakpoints, there is a value $\varepsilon' > 0$ such that $t_0$ is linear on the intervals $[0, \varepsilon']$ and $[1-\varepsilon', 1]$. Let $\varepsilon = \mathbf{min}\{\varepsilon', \varepsilon''\}$, giving us a value with both of these properties. 

Since $\sqrt{F'_\tau} \rtimes \langle t_0 \rangle = G \cong H*_{t, \sigma}$, we can say that $\sqrt{F'_{\tau}} = \bigcup_{n \geq 1} t^n H t^{-n}$. Hence for each $f \in \sqrt{F'_{\tau}}$, there is a value $n$ such that $t^{-n} f t^n \in H$, hence $t^{-n} f t^n$ is supported in $[\varepsilon, 1-\varepsilon]$. From here, we can see that any $f \in \sqrt{F'_{\tau}}$ must be supported in $[t_0^n(\varepsilon), t_0^n(1-\varepsilon)]$ for some $n$. As $\sqrt{F'_{\tau}}$ has support $(0,1)$, we can see that $(t_0^n(\varepsilon))_{n\in\N}$ must have a subsequence that converges to $0$ and $(t_0^n(1-\varepsilon))_{n\in\N}$ must have a subsequence that converges to $1$. As $t_0$ is linear on the intervals $[0, \varepsilon]$ and $[1-\varepsilon, 1]$, it holds $t_0(\varepsilon) < \varepsilon$ and $t_0(1-\varepsilon) > 1-\varepsilon$. Hence $t_0$ must have slope smaller than $1$ near $0$ and slope bigger than $1$ near $1$. Therefore $ab <0$. Given that we started with the assumption that $G=\ker(\chi)$ was of type $\tFPm2$, we obtain the implication 
\[ \chi = a\lambda + b\rho \quad \text{ and } \quad \ker(\chi) \text{ of type } \tFPm2 \implies ab < 0 \]
whenever $a,b \in \Q\setminus\{0\}$. The contrapositive of this is that $ab>0$ implies $\ker(\chi)$ is not of type $\tFPm2$. Combining this with \cref{kernel}, we see that we cannot have both $[\chi]$ and $[-\chi]$ in $\Sigma^2(F_{\tau},R)$. In particular, if the antipodal point $[-\chi]$ lies in $\Sigma^2(F_\tau,R)$, then by \cref{prop:almosteverything} we have that $[\chi] \not\in \Sigma^2(F_{\tau},R)$.

Transfering this result to the homotopical invariant with the use of \cref{homotopy}, we conclude that if $[\chi] \not\in \Sigma^2(F_{\tau},R)$, then $[\chi] \not\in \Sigma^2(F_{\tau})$. 
\end{proof}

This finishes off the proof of \cref{main-thm}. \qed

\bibliographystyle{plain}
\bibliography{references}

\end{document}